\documentclass[12pt]{amsart}

\usepackage{latexsym,color,amsmath,amsthm,amssymb,amscd,amsfonts}
\usepackage{tikz}
\usetikzlibrary{arrows}
\usepackage{scalefnt}
\usepackage[font=small,labelfont=bf]{caption}
\usepackage{float}
\usepackage{graphicx}
\usepackage{relsize}
\usepackage{amsopn}
\usepackage{mathrsfs}  %
\usepackage[hidelinks]{hyperref}
\usepackage{bm}
\usepackage{bbm}
\usepackage[shortlabels]{enumitem}
\usepackage{multicol}
\usepackage{float} 
\usepackage{todonotes}
\usepackage{cite}
\usepackage[margin=2.5cm]{geometry}

\newtheoremstyle{nonum}{}{}{\itshape}{}{\bfseries}{.}{ }{\thmnote{#3}}

\newtheorem{thm}{Theorem}[section]%
\newtheorem*{thm*}{Theorem}

\newtheorem{cor}[thm]{Corollary}
\newtheorem{lem}[thm]{Lemma}
\newtheorem{prop}[thm]{Proposition}
\newtheorem{rem}[thm]{Remark}
\newtheorem{conj}{Conjecture}

\newtheorem*{definition*}{Definition}

\newtheorem*{rems*}{Remarks}

\theoremstyle{nonum}

\newcommand{\R}{\mathbb R}
\newcommand{\RR}{\mathbb R}
\newcommand{\RRn}{\RR_+}

\def\Vol{{\rm Vol}}

\newcommand{\iprod}[2]{\langle #1,#2 \rangle} %

\def\vol{{\rm Vol}}

\def\conv{{\rm conv}}

\newtheoremstyle{nonum}{}{}{\itshape}{}{\bfseries}{.}{ }{\thmnote{#3}}

\newtheorem*{lem*}{Lemma}

\theoremstyle{remark}

\newcommand{\Gnc}{\mathcal{G}_n^c}
\newcommand{\Gnjc}{\mathcal{G}_{n,j}^c}

\def\Vol{{\rm Vol}}

\def\conv{{\rm conv}}

\newcommand{\dual}{^\circ}



\begin{document}

\title {Geometric Inequalities for Anti-Blocking Bodies}

\author{Shiri Artstein-Avidan}
\author{Shay Sadovsky}
\address{School of Mathematical Sciences, Tel Aviv University, Israel}
\email{shiri@tauex.tau.ac.il}
\email{shaysadovsky@mail.tau.ac.il}

\author{Raman Sanyal}
\address{Institut f\"ur Mathematik, Goethe-Universit\"at Frankfurt, Germany} 
\email{sanyal@math.uni-frankfurt.de}

\maketitle
\begin{abstract}
    We study the class of (locally) anti-blocking bodies as well as some
    associated classes of convex bodies. For these bodies, we prove geometric
    inequalities regarding volumes and mixed volumes, including Godberson's
    conjecture, near-optimal bounds on Mahler volumes, Saint-Raymond-type
    inequalities on mixed volumes, and reverse Kleitman inequalities for mixed
    volumes. We apply our results to the combinatorics of posets and prove
    Sidorenko-type inequalities for linear extensions of pairs of
    $2$-dimensional posets. 
    
    The results rely on some elegant decompositions of differences of
    anti-blocking bodies, which turn out to hold for anti-blocking bodies with
    respect to general polyhedral cones.
\end{abstract}

\section{Introduction and results}\label{sec:intro}
\newcommand\Def[1]{\textbf{#1}}%
\newcommand\PM{\{-1,1\}}%

A convex body $K \subseteq\RR^n$ is called \Def{$1$-unconditional} if
$(x_1,\dots,x_n) \in K$ implies that the points $(\pm x_1,\dots, \pm x_n)$
also belong to $K$. Clearly such a body is determined by its subset  $K_+ := K
\cap \RRn^n$, which is itself a convex set of a special kind, called an
``anti-blocking body''~\cite{Fulkerson1971blocking, Fulkerson1972anti} or a
``convex corner''~\cite{Bollobas-Reverse-Kleitman}. A convex body $K \subset
\RRn^n$ is called \Def{anti-blocking} if for all $y \in K$ and $x \in \RRn^n$
with $x_i \le y_i$ for all $i=1,\dots,n$, it holds that $x \in K$.  It is
clear that $1$-unconditional bodies and anti-blocking bodies are in one-to-one
correspondence.

The class of \emph{locally} anti-blocking bodies is defined to be those bodies
in $\RR^n$ such that all of their orthants are anti-blocking. More formally,
for $\sigma \in \PM^n$ and $S \subseteq\R^n$ let us write $\sigma S = \{ (
\sigma_1 x_1, \dots, \sigma_n x_n) : x \in S \}$. We define a convex body $K
\subseteq\RR^n$ to be \Def{locally anti-blocking} if $(\sigma K) \cap \RRn^n$
is anti-blocking for all $\sigma \in \PM^n$. In particular, the distinct bodies
$K_\sigma := K \cap \sigma \RRn^n$ form a \Def{dissection} of $K$.  That is,
$K = \bigcup_\sigma K_\sigma$ and the $K_\sigma$ have disjoint interiors. By
definition, $1$-unconditional convex bodies are precisely the locally
anti-blocking bodies with $\sigma K_\sigma = K_+$ for all $\sigma$.

In this paper we will investigate the structure of locally anti-blocking
bodies and use it to show various volume and mixed volume inequalities for
this class of bodies. As we shall explain in Section~\ref{sec:ab-bodies}, this
class is closed under convex hull, intersections, Minkowski addition, and
polarity.  The dissection presented yields good descriptions of their (mixed)
volumes. We will prove some well-known conjectures for this special class of
convex bodies. Two specific instances of locally anti-blocking bodies which
shall play a central role in this paper are the following: Let $K_1, K_2
\subset \RRn^n$ be anti-blocking bodies. Then
\begin{equation}\label{eq:main_ex}
    K_1 - K_2 = \{ p_1-p_2 : p_1 \in K_1, p_2 \in K_2\}  \quad \text{ and }
    \quad K_1 \vee (-K_2) := \conv( K_1 \cup -K_2 )
\end{equation}
are locally anti-blocking bodies. These constructions have occurred in the
context of geometric combinatorics; see~\cite{Sanyal-ABbodies} and references therein. We shall also use anti-blocking bodies to build the so-called $C$-bodies, which arose in works of Rogers and Shephard, and which are themselves not locally-anti-blocking but share some of their structural properties. 

\subsection{Godberson conjecture}
A natural way to construct a centrally-symmetric convex body from an arbitrary
convex body $K$ is by its \emph{difference body}, $K-K$. The volume of the
difference body was studied by Rogers and Shephard in~\cite{RS-difference},
and bounded from above by $\binom{2n}{n}\vol(K)$ with equality only for
simplices. The volume of the difference body may also be interpreted in terms
of mixed volumes of $K$ and $-K$. The mixed volume, introduced by Minkowski,
is a non-negative symmetric function $V$ defined on $n$-tuples of convex
bodies in $\RR^n$ that satisfies
\begin{equation}
    \vol_n(\lambda_1 K_1 + \dots + \lambda_n K_n) = \sum_{i_1,\dots, i_n =1}^n
    \lambda_{i_1} \cdots \lambda_{i_n} V(K_{i_1},\dots, K_{i_n}).
\end{equation}
for any convex bodies $K_1,\dots, K_n \subset \RR^n$ and $\lambda_1, \dots,
\lambda_n \ge 0$; see~\cite{Schneider-book-NEW} for more on this.  We call
$V(K_1,\dots, K_n)$ the \Def{mixed volume} of $K_1,\dots, K_n$. We write
$V(K[j], T[n-j])$ to denote the mixed volume of two bodies $K$ and $T$, where
$K$ appears $j$ times and $T$ appears $n-j$ times.

In light of the description of mixed volumes, we may restate Rogers and Shephard's result as follows:
\[
    \vol(K-K) \ = \ \sum_{j=0}^n \binom{n}{j}V(K[j],-K[n-j]) \ \leq \ \sum
    \binom{n}{j}^2 \vol(K) \ = \ \binom{2n}{n}\vol(K).
\]
Godbersen conjectured that the inequality is in fact term-by-term: 
\begin{conj}[Godberson~\cite{Godbersen}]
    For any convex body $K \subseteq\RR^n$ and $0 < j < n$
    \begin{equation}\label{eqn:Godberson}
        V(K[j],-K[n-j]) \ \leq \ \binom{n}{j} \vol (K).	 
    \end{equation}
    Equality holds if and only if $K$ is a simplex.
\end{conj}
Despite many efforts, the best known bound for $V(K[j],-K[n-j])$ to date is
$\frac{n^n}{j^j (n-j)^{(n-j)}}\vol(K)$, differing from the conjectured bound
by up to $\sqrt{n}$, see \cite{ArtsteinKeshet} and \cite{Artstein-Note}.
The conjecture in full is known only for very specific classes of bodies such
as bodies of constant width \cite{Godbersen}. Our first result is that for
anti-blocking bodies, the conjecture holds.

\begin{thm}\label{thm:god-for-closed-down}
    Anti-blocking bodies satisfy Godberson's conjecture.
\end{thm}

The proof for the theorem makes use of the fact that the difference bodies of
anti-blocking bodies are locally anti-blocking and is explained in Section
\ref{sec:God}. Among anti-blocking bodies, the only equality case is indeed
the simplex. Coupled with the fact that the only anti-blocking bodies with a
center of symmetry are axes-parallel boxes, we get that the quantity $ \vol
(K) /V(K[j],-K[n-j]) $ serves as a measure of symmetry in the sense of
Gr\"unbaum \cite{Grunbaum} for the class of anti-blocking bodies. 

\subsection{Mahler conjecture}
Another family of results we discuss revolves around Mahler's conjecture for
centrally-symmetric bodies. We let $K\dual$ denote the \Def{polar} (or
\Def{dual}) of a convex body $K$
\[ 
    K^\circ = \{y\in \RR^n:  \iprod{x}{y}\le 1 \text{ for all } x \in K\}.
\]
A long standing open problem in convex geometry is to bound tightly from below
the \Def{volume product} $\vol(K)\vol(K^\circ)$ (sometimes called the
Santal\'o or Mahler volume) of a convex body $K$ containing the origin in
its interior. This quantity, which is a linear invariant, has many
different interpretations, see e.g.~\cite{ArtsteinOstrover-symplectic,
BourgainMilman}.  In the general case, where no central symmetry is assumed,
one usually assumes that $K$ is {\bf{centered}}, namely that the center of
mass of $K$ is at the origin (it is easy to check that the center of mass of
$K^\circ$ is the origin, if one takes the infimum of the volume product over
translations of $K$). The Mahler volume  of a centered convex body is bounded
from above by $\vol(B_2^n)^2$ \cite{Santalo, blaschke1917affine}, and is
conjectured to be bounded from below by its value on a regular simplex (in the
general case) and by its value on a cube (in the centrally-symmetric case).
We will be concerned with the latter, which can be written as

\begin{conj}[Mahler Conjecture for centrally-symmetric bodies]\label{conj:Mahler}
    If $K \subseteq\R^n$ is  a centrally-symmetric convex body with non-empty interior
    then
    \begin{equation}\label{eqn:Mahler}
        \vol(K) \cdot \vol(K\dual) \ \ge \  \frac{4^n}{n!}.
    \end{equation}
\end{conj}

Mahler's conjecture is known to hold in dimensions $n\le 3$, see
\cite{IriyehShibata} and \cite{FradeliziHubardMeyerPensadoZvavitch}. It is
also known to hold for specific classes of bodies, for example for 
zonotopes~\cite{r-zonoids}.

Saint-Raymond~\cite{Saint-Raymond} proved Mahler's conjecture for
$1$-unconditional convex bodies, essentially by proving the equivalent form
for anti-blocking bodies. For $K \subseteq\RR^n$ anti-blocking, define 
\begin{equation}\label{eqn:AK}
    AK \ := \ \{ y \in \RRn^n : \iprod{y}{x} \le 1  \text{ for all } x \in K
    \} \, .
\end{equation}
It is easy to see that $AK$ is again anti-blocking and $A(AK) = K$.
Saint-Raymond proved that if $K$ is a full-dimensional anti-blocking body,
then
\begin{equation}\label{eqn:SaintRaymond}
    \vol(K) \cdot \vol(AK) \ \ge \ \frac{1}{n!} \, .
\end{equation}
If $K$ is $1$-unconditional, then $(K\dual)_+ = AK_+$ and $\vol(K) = 2^n
\vol(K_+)$.  It is not hard to check that Saint-Raymond's result implies that
Conjecture \ref{conj:Mahler} in fact holds for the larger class of
locally anti-blocking bodies, and without central symmetry assumption;
cf.~Corollary~\ref{cor:mahler-for-locally-ab}

It turns out that considering anti-blocking bodies instead of $1$-unconditional ones, gives rise to a strengthening of Saint-Raymond's inequality which was not observed before:

\begin{thm}[Saint-Raymond-type inequality for Mixed
    Volumes]\label{thm:mahler-for-vj}
    For $0 \le j \le n$, let $K  \subset
    \RRn^n$ be anti-blocking, then 
    \begin{equation}\label{eq:mixed-sr-simple2}
    V(K[j],-K[n-j]) \cdot V(AK[j],-AK[n-j]) \ \ge \ \frac{1}{j!(n-j)!} \,.
    \end{equation}
    Moreover, for $K,T \subseteq\RRn^n$ anti-blocking,
    \begin{equation}\label{eq:mixed-sr-simple}
    V(K[j],-T[n-j]) \cdot V(AK[j],-AT[n-j]) \ \ge \ \frac{1}{j!(n-j)!} \, .
    \end{equation}
    In even higher generality,         
    let $K_1, \dots, K_j, T_1, \dots, T_{n-j} \subset
    \RRn^n$ be anti-blocking bodies. Then,
    \begin{equation}\label{eq:mixed-sr-comlicated}
    V(K_1,\dots, K_j, -T_1,\dots -T_{n-j}) \cdot V(AK_1,\dots AK_j, -AT_1,\dots
    -AT_{n-j}) 
	\ \ge \ \frac{1}{j!(n-j)!} \, .
\end{equation}
\end{thm}

Our second result in the direction of Mahler's conjecture has to do with another special class of bodies.
In~\cite{RS-sections-projections}, Rogers and Shephard introduced a centrally symmetric convex body associated to every  convex body $K \subseteq\RR^n$, named the \Def{$C$-body} of $K$, 
\[
    C(K) \ := \ \conv( K \times \{1\} \cup -K \times \{-1\} ) \subset
    \RR^{n+1} \, .
\]
We shall refer to $K$ as the \Def{base} of $C(K)$. Although probably by
coincidence, the name \emph{$C$-body} is suitable, as the construction
coincides with the \emph{Cayley construction}  of $K$ and $-K$;
see~\cite{DLRS}. 

With our understanding of anti-blocking bodies, we prove a lower bound on the
Mahler volume of $C$-bodies with an anti-blocking base $K$, which differs from
the conjectured bound by a factor of order $\sqrt{n}^{-1}$. Let us mention
that proving Mahler's Conjecture~\ref{conj:Mahler} for \emph{all} $C$-bodies
implies the conjecture for all convex bodies. Also, the class of $C$-bodies of
anti-blocking bodies includes linear images of both the cube and its dual, the
crosspolytopes. 
We show 
\[
    \vol(C(K)) \cdot \vol(C (K)\dual) \  \ge \ \frac{\sqrt{2\pi n}}{n+1}
    \cdot \frac{4^{n+1}}{(n+1)!}.
\] 	

In fact we will show this lower bound for the Mahler product of bodies in a
somewhat wider class of bodies. Given an anti-blocking body
$K\subseteq\RRn^n$, define a $1$-parameter family of centrally-symmetric body in $\RR^{n+1}$, for  $0 \le \lambda \le
1$,   by 
\[
    C_\lambda(K) \ := \ \conv( K \times \{\lambda\} \cup -K \times
    \{-\lambda\} \cup [-\mathbf{e}_{n+1}, \mathbf{e}_{n+1} ] ) \subset
    \RR^{n+1} \, ,
\]
where $\mathbf{e}_{n+1} = (0,\dots,0,1)$. Note that when $\lambda = 0$ the
body is locally anti-blocking, so it must satisfy Mahler's conjecture.  For
$\lambda = 1$, the body $C_1(K) = C(K)$. 
\begin{thm}\label{thm:mahler-for-c-bodies}
    If $K \subseteq\RRn^n$ is anti-blocking, then
    \[
        \vol(C_\lambda(K)) \cdot \vol(C_\lambda(K)\dual) \  \ge \ \frac{\sqrt{2\pi n}}{n+1} \cdot
        \frac{4^{n+1}}{(n+1)!}
    \] 	
    for all $0 \le \lambda \le 1$.
\end{thm}

We show that the volume of $C_\lambda(K)$ is independent of $\lambda$. The
result then follows from the convexity of reciprocals of volumes of dual
bodies along shadow systems and the following interesting inequality.

\begin{prop} \label{prop:nearly-mahler-1}
	For a full-dimensional anti-blocking $K\subseteq\RR_{+}^n$ we have 
	\begin{equation*}
	\vol(K \vee -K ) \cdot \vol(AK \vee -AK) \ \ge \
    \frac{1}{n!}\left(\sum_{j=0}^n \binom{n}{j}^{\frac{1}{2}}\right)^2 \, . %
	\end{equation*}
\end{prop}

\subsection{Reverse Kleitman inequalities}
Our final volume inequality is in a somewhat opposite direction  -- we bound a
Godbersen-type quantity of mixed volumes from below. For $V(K[j], -K[n-j])$
the lower bound $\vol(K)$ is trivial by the Brunn-Minkowski inequality and its
consequences.  When $-K$ is replaced by $-T$,  we may replace the bound by
$\vol(K)^{j/n}\vol(T)^{(n-j)/n}$. When in addition both bodies are
anti-blocking, a stronger inequality of this form holds, and we show the
following inequality, which is motivated by the Reverse Kleitman inequality of
Bollob\'as, Leader, and Radcliffe~\cite{Bollobas-Reverse-Kleitman}:

\begin{thm}\label{thm:kleitman-bound-on-mixed}
	Given two anti-blocking bodies, $K, T \subseteq\RR_+^n$,
	\[ 
        V(K[j],T[n-j]) \ \leq \  V(K[j],-T[n-j]).
    \]
    In particular, 
    \[ \vol(K+T) \ \le \ \vol(K-T). \] 
\end{thm}

\subsection{Linear extensions}
In Section~\ref{sec:posets} we apply our geometric inequalities to the study
of linear extensions of partially ordered sets.  Sidorenko~\cite{Sidorenko}
and later Bollob\'as, Brightwell, Sidorenko~\cite{BBS} established a negative
correlation result for linear extensions of $2$-dimensional posets. We
generalize their results to $2$-dimensional double posets that gives a refined
(mixed) Sidorenko inequality for $2$-dimensional posets. 

\subsection{$C$-anti-blocking bodies}
An important ingredient in many of our proves is the following decomposition
for anti-blocking bodies $K, K' \subset \RRn^n$
\begin{equation}\label{eqn:can-ab}
    K - K' \ = \  \bigcup_{ E} P_E K  \times P_{E^\perp}(-K') \, 
\end{equation}
where $K$ ranges over all coordinate subspaces and $P_E$ and $P_{E^\perp}$ are
orthogonal projections onto $E$ and $E^\perp$, respectively. The geometry of
anti-blocking bodies is intimately tied to the convex cone $\RRn^n$.  In
Section~\ref{sec:C-ab}, we develop the basics of \Def{$C$-anti-blocking}
bodies, where $C$ is a general polyhedral and pointed cone and prove a
generalization of~\eqref{eqn:can-ab}.

\textbf{Acknowledgements.} The collaboration leading to this manuscript
started while the authors were in residence at the Mathematical Sciences
Research Institute in Berkeley, California, during the fall semester of 2017.
The first two authors participated in the program ``Geometric Functional
Analysis and Applications'' and the third named author in the program
``Geometric and Topological Combinatorics'', which were, luckily, carried out
in parallel.  The first and second named authors  received funding from the
European Research Council (ERC) under the European Union’s Horizon 2020
research and innovation programme (grant agreement No 770127).

\section{Anti-blocking and locally anti-blocking bodies}\label{sec:ab-bodies}

\subsection{Anti-blocking bodies}
Define a partial order on $\RR^n$ by 
\[
    x \leq y \quad {\rm if~and~only~ if} \quad x_i \ \leq \ y_i \  \text{ for all } i =
    1,\dots,n \, .
\]
A set $S \subseteq\RR^n$ is called \textbf{order-convex} if for any $y,x\in S$
and $z \in \R^n$ with $y \leq z \leq x$, we have $z \in S$. A convex body $K
\subset \RRn^n$ is called \textbf{anti-blocking} if $K$ is order-convex and $0
\in K$. Fulkerson~\cite{Fulkerson1971blocking, Fulkerson1972anti} studied
anti-blocking polytopes in the context of combinatorial optimization. The name
is derived from the notion of (anti-)blocking in hereditary set systems. These
polytopes have been discovered in various contexts and are sometimes called
\textbf{convex corners} (cf.  \cite{Bollobas-Reverse-Kleitman}). As explained
in the introduction, they are in one-to-one correspondence with
$1$-unconditional convex bodies.  Figure~\ref{fig:anti-blocking} shows some
three-dimensional examples.

Let $\Gnc$ denote the collection of coordinate subspaces in $\RR^n$. For $E
\in \Gnc$, we write $P_E : \RR^n \to E$ for the orthogonal projection onto
$E$.  Anti-blocking convex bodies can be equivalently defined as those convex
bodies $K \subseteq\RRn^n$ for which 
\[
    P_E K \ = \  K \cap E
\] 
holds for all $E \in \Gnc$.

\begin{prop}\label{prop:AB-ops}
    Let $K_1, K_2 \subseteq\RRn^n$ be anti-blocking bodies. Then $K_1 \cap
    K_2$, $K_1 \vee K_2$, and $K_1 + K_2$ are anti-blocking bodies.
\end{prop}
\begin{proof}
    We note
    \[
        P_E(K_1 + K_2) \ = \ P_E K_1 + P_E K_2 \ = \ (E \cap K_1) + (E \cap
        K_2) \ = \ E \cap (K_1 + K_2) \, ,
    \]
    where the last equality follows from the fact $K_1 \cap E$ and $K_2
    \cap E$ are faces.
    The claim
    for $K_1 \cap K_2$ is clear and the argument for $K_1 \vee K_2$ is
    similar. 
\end{proof}

\begin{figure}[h]
	\centering
	
	\begin{tikzpicture}[scale = 1.2]
	\draw[-] (0,0,0) -- (1.5,0,0) node[right] {$ $};
	\draw[-] (0,0,0) -- (0,1.5,0) node[above] {$ $};
	\draw[-] (0,0,0) -- (0,0,1.5) node[below left] {$ $};
	
	\shadedraw [color=white, fill=magenta!90,opacity=0.7] (0,0,1.5) -- (0,1.5,0) -- (1.5,0,0);
	
	\end{tikzpicture} \qquad
	\begin{tikzpicture}[scale = 1.2]
	\draw[-] (0,0,0) -- (1.5,0,0) node[right] {$ $};
	\draw[-] (0,0,0) -- (0,1.5,0) node[above] {$ $};
	\draw[-] (0,0,0) -- (0,0,1.5) node[below left] {$ $};
	
	\shadedraw [right color=magenta!20!white!77!gray, left color=magenta!5!gray, color = magenta!50!black, shading angle=0, opacity=0.9] (0,1.5,0)--(0,1.5,1.5)--(1.5,1.5,1.5)--(1.5,1.5,0)--(0,1.5,0);
	\shadedraw [right color=magenta!20!white!77!gray, left color=magenta!5!gray, color = magenta!50!black, shading angle=-90, opacity=0.9] (1.5,0,0) -- (1.5,0,1.5) -- (1.5,1.5,1.5) -- (1.5,1.5,0) -- (1.5,0,0);
	\draw [magenta!50!black, fill=white!75!magenta!,opacity=0.9] (0,0,1.5) -- (1.5,0,1.5) -- (1.5,1.5,1.5) -- (0,1.5,1.5) -- (0,0,1.5);
	\end{tikzpicture}\qquad
	\begin{tikzpicture}[scale = 1.2]
	\draw[-] (0,0,0) -- (1.5,0,0) node[right] {$ $};
	\draw[-] (0,0,0) -- (0,1.5,0) node[above] {$ $};
	\draw[-] (0,0,0) -- (0,0,1.5) node[below left] {$ $};
	
	\shade[ball color=magenta!40!gray!80!orange,opacity=0.7] (1.5,0) arc (0:90:1.5) {[x={(0,0,1)}] arc (90:0:1.5)} {[y={(0,0,1)}] arc (90:0:1.5)};
	\end{tikzpicture}
	\caption{Three examples of anti-blocking bodies in $\R^3$ -- a
    right-angled simplex, a cube and an intersection of the ball with the
    orthant, $B_2^3 \cap \RR_+^3$.}\label{fig:anti-blocking}
\end{figure}
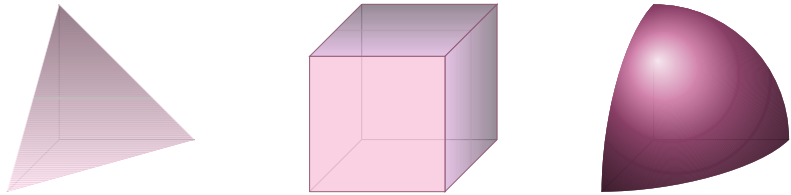

\subsection{Locally anti-blocking bodies}
Recall that for $\sigma \in \PM^n$ and $S \subseteq\RR^n$, we write $\sigma S =
\{ (\sigma_1x_1,\dots,\sigma_nx_n) : x \in S \}$. We write $\RR^n_\sigma :=
\sigma \RRn^n$ and we set $S_\sigma := S \cap \RR^n_\sigma$. For any set $S$
this gives a decomposition
\[
    S \ = \ \bigcup_\sigma S_\sigma \, ,
\]
where $S_\sigma \cap S_{\sigma'}$ is of measure zero whenever $\sigma \neq
\sigma'$.  We say that a convex body $K \subseteq\RR^n$ is \Def{locally
anti-blocking} if $\sigma K_\sigma$ is anti-blocking for every $\sigma \in
\PM^n$. Note that $K$ is $1$-unconditional if and only if  $K_\sigma = \sigma
K_+$ for all $\sigma$, where $K_+:= K\cap \RR^n_+$.

\begin{lem} \label{lem:locally-ab-sum-decomp}
    Let $K, K' \subseteq\RR^n$ be locally anti-blocking bodies. Then $K \cap K'$, $K
    \vee K'$, and $K + K'$ are locally anti-blocking as well.
\end{lem}
\begin{proof}
    The statement for $K \cap K'$ is clear. To show that 
    \[
        K + K'  \ = \ \bigcup_\sigma K_\sigma + K'_\sigma \, 
    \]
    note that $K_\sigma + K'_\sigma \subseteq K + K'$ for all $\sigma\in
    \PM^n$. To show the other inclusion, pick $y \in K$ and $y' \in K'$.  We
    would like to find an orthant $\tau \in \PM^n$ and $x \in K_\tau$, $x\in
    K'_\tau$, with $x + x' = y + y'$. To this end, identify first
    $\sigma,\sigma' \in \PM^n$ such that $\sigma y, \sigma'y' \in \RRn^n$, and
    set $I(y,y') := \{ i : y_i y_i' < 0 \}$.
    We let $x_i = y_i, x'_i = y'_i$, for $i\not\in I(y,y')$. For $i\in
    I(y,y')$ we distinguish two cases, if $(\sigma y)_i> (\sigma' y')_i$, we
    let $x_i =  y_i+y'_i$ and $x'_i = 0$, and otherwise we let $x_i = 0$ and
    $x_i' = y_i+y'_i$. 
    
    Clearly the new vectors $x,x'$ satisfy that they are in a single orthant
    (as we have changed one of the coordinates to $0$ in the orthants where
    the original vectors differed) and, by construction, they satisfy $x + x'
    = y + y'$. Finally, since $K_+$ and $K'_+$ are locally anti-blocking, and
    we made sure that $0 \le \sigma x \le \sigma y$ and $0 \le \sigma' x' \le
    \sigma'y'$, we get that  $x\in K$ and $x'\in K'$, which proves the claim. 

    To prove the statement for the convex hull $K \vee K'$, note that
    \[
        K \vee K' \ = \ 
        \bigcup_{0 \le \lambda \le 1}
        \lambda K + (1 - \lambda) K'
        \ = \ 
        \bigcup_{0 \le \lambda \le 1}
        \bigcup_\sigma \lambda K_\sigma + (1 - \lambda) K'_\sigma
        \ = \ 
        \bigcup_\sigma K_\sigma \vee  K'_\sigma \, , 
    \]
    where the second equality follows from the statement about Minkowski sums.
\end{proof}

Recall that for a convex body $K \subseteq\RR^n$, its \Def{polar} is defined as
\[
    K\dual \ := \ \{ y \in \RR^n : \iprod{y}{x} \le 1 \text{ for all } x \in K
    \} \, .
\]
The polar is again a convex body if the origin belongs to the interior of $K$, otherwise it is an unbounded closed convex set. 
Polarity for anti-blocking bodies was defined in~\eqref{eqn:AK}. It can be
rewritten as $AK  =  (K - \RRn^n)\dual = K^\circ \cap \RRn^n$.  We extend the
definition to all orthants by setting $A_\sigma K := \sigma A(\sigma K)$ for
$\sigma \in \PM^n$.  The next result states that the class of locally
anti-blocking bodies is closed under taking polars.

\begin{lem}\label{lem:polarity-AB-decomposes}
    Let $K \subseteq \RR^n$ be a locally anti-blocking body.  Then $K\dual$ is
    locally anti-blocking with decompostion
    \[
        K\dual \ = \ \bigcup_{\sigma} A_\sigma K_\sigma \, .
    \]
\end{lem}
\begin{proof}
    For $\sigma, \sigma' \in \PM^n$, let $E$ be the coordinate subspace
    spanned by $\RR^n_\sigma \cap \RR^n_{\sigma'}$. It follows that for $y \in
    \R^n_\sigma$ and $x \in \R^n_{\sigma'}$ we have $\iprod{y}{x} \le
    \iprod{y}{P_E(x)}$. With the anti-blocking property for each
    $K_{\sigma'}$, this implies 
    \[
        \max_{x\in K} \, \iprod{y}{x} \ = \ \max_{x \in K_\sigma} \, \iprod{y}{x}  \
        .
    \]
    This shows that $y \in K\dual$ if and only if $y \in A_\sigma K_\sigma$,
    which proves the claim.
\end{proof}

\subsection{Decompositions}
The following  decompositions are the main tools for our results.  The first
one was shown in~\cite{Sanyal-ABbodies}.  Let $K, K' \subseteq \RRn^n$ be
anti-blocking bodies. Since anti-blocking implies locally anti-blocking,
Lemma~\ref{lem:locally-ab-sum-decomp} yields that
\[
    K - K' \ = \  K + (-K') \quad \text{ and } \quad
    K \vee (-K)
\]
are both locally anti-blocking. Let $\sigma \in \PM^n$ and let $E$ be the
coordiante subspace spanned by $\RRn^n \cap \RR^n_\sigma$. Then $K \cap
\RR^n_\sigma = P_E K $ and $(-K') \cap \RR^n_\sigma = P_{E^\perp}(-K)$
together with the decomposition given in Lemma~\ref{lem:locally-ab-sum-decomp}
yields the following fundamental decompositions.

\begin{lem}\label{lem:decomposition-sum}
    Let $K, K' \subseteq \RRn^n$ be anti-blocking convex bodies. Then
    \begin{align}\label{eq:vol-of-(K-tK)}
        K + (-K') \ &= \   \bigcup_{ E \in \Gnc} P_E K  \times
        P_{E^\perp}(-K')  \\ 
        K \vee (-K') \ &= \ \bigcup_{E\in \Gnc} P_E K  \vee P_{E^\perp}(-K')
        \, . 
	\end{align}
\end{lem}

One benefit of these decompositions is that we get the following simple
formulas for (mixed) volumes, the proofs of which are immediate by taking the
volume of $K\vee \lambda K'$ or $K - \lambda K'$ using the equations from the previous lemma, and comparing the coefficients of $\lambda^i$ on both sides.

\begin{cor}\label{cor:decomposition-conv}
    Let $K, K' \subseteq \RRn^n$ be anti-blocking bodies. Then
    \begin{align}\label{eq:jmixformula}
	    V(K[j], -K'[n-j]) \ &= \ {\binom{n}{j}}^{-1}\sum_{E \in \Gnjc}
        \vol_j(P_E K) \cdot \vol_{n-j} (P_{E^\perp}K') \\
	\vol(K \vee -K') \ &= \  \sum_{j=0}^n   V(K[n-j],-K'[j])
\label{eq:vol-j-conv(K,-t)} \, .
	\end{align}
\end{cor}

\section{Around Godbersen's conjecture}\label{sec:God}

The proof of Theorem \ref{thm:god-for-closed-down} is immediate from the formula for the mixed volume of two anti-bocking bodies, joined with the following Lemma proved in~\cite{RS-sections-projections} by Rogers and
Shephard.

\begin{lem}[Rogers--Shephard Lemma for Sections and Projections]\label{thm:rogers-shephard-sections-projections}
        Let $K\subseteq\RR^n$ be a convex body, and let $E \subseteq \RR^n $ be
        a $j$-dimensional subspace. Then 
		\begin{equation}\label{eq:RSlemm}
		\vol_{j}\left(K\cap E\right) \cdot \vol_{n-j}\left(P_{E^\perp}K\right)
        \ \le \ 	\binom{n}{j}\vol_{n}\left(K\right).
		\end{equation}
		Moreover, equality holds if and only if for every direction $v\in E^\perp$, the intersection of
		$K$ with $E+\R_+v$ is obtained by taking the convex hull of $K\cap E$ and one more point.
\end{lem}

There are many different proofs for Lemma \ref{thm:rogers-shephard-sections-projections}. Let us include a simple one which will be useful for us to observe. 

\begin{proof}
Given a convex body $K$, and a $j$-dimensional subspace $E$, use Schwartz symmetrization with respect to $E$ to replace, for any $x\in E^\perp$ the section $(x+E)\cap K$ with a Euclidean ball in $E$ of the same volume centered at $x$. The new body, $\widehat{K}$, has the same volume as $K$ by Fubini's theorem, and it satisfies $\vol(K\cap E) = \vol(\widehat{K}\cap E)$ and $\vol(P_{E^\perp}(K)) = \vol(\widehat{K}\cap E^\perp)$. By Brunn's concavity principle, $\widehat{K}$ is convex.

Using the fact that $(\widehat{K}\cap E) \vee (\widehat{K}\cap E^\perp)
\subseteq \widehat{K}$, and that the volume of a convex hull of two convex sets
$A\subseteq E,B\subseteq E^\perp$ which includes the origin is simply $\vol(A)
\vol(B) /\binom{n}{j}$, we end up with 
\[ \vol(K) = \vol(\widehat{K}) \ge \frac{1}{\binom{n}{j}} \vol(\widehat{K}\cap E) \vol(\widehat{K}\cap E^\perp) =\frac{1}{\binom{n}{j}} \vol( \widehat{K}\cap E) \vol(P_{E^\perp} \widehat{K} ). \] 

The equality case is precisely when $\widehat{K} = (\widehat{K} \cap E) \vee
(\widehat{K} \cap E^\perp)$.
\end{proof}

\begin{rem}\label{rem:equality-RS-antiB}
    Inspecting the proof we see that when $K$ is anti-blocking and $E$ is a
    coordiante subspace, the first step of Schwartz symmetrization is
    redundant, namely we may use $\widehat{K}=K$ in the proof. This implies
    that the equality case in the Rogers--Shephard Lemma, when $K$ is
    anti-blocking, is precisely when  $ {K} = (K \cap E) \vee ( K \cap
    E^\perp )$.
\end{rem}

The following result, joint with the aforementioned equality condition, 
 will yield the equality condition for Godbersen's
conjecture in the case of anti-blocking bodies.
\begin{prop}\label{prop:RS-AB-equality}
    Let $K \subseteq \RRn^n$ be a full-dimensional anti-blocking body  and $1
    \le j < n$. If $K = (K \cap E) \vee (K \cap E^\perp)$ for every
    $j$-dimensional coordinate subspace $E$, then $K$ is a simplex.
\end{prop}
\begin{proof}
    It suffices to show that every extreme point of $K$ is of the form $\mu_i
    e_i$ for $i = 1,\dots,n$ and $\mu \ge 0$. By assumption, the extreme
    points of $K$ are contained in $E \cup E^\perp$ for every $E$. Assume that
    $p \in K$ is an extreme point such that $p_r > 0$ and $p_s > 0$ for some
    $1 \le r < s \le n$.  Now choose $E$ such that $e_r \in E$ and $e_s
    \not\in E$. Then $p \not\in E \cup E^\perp$, which proves the claim.
\end{proof}

We are now in the position to 
prove Godbersen's conjecture
 for anti-blocking bodies.

\begin{proof}[Proof of Theorem~\ref{thm:god-for-closed-down}]
    By equation (\ref{eq:jmixformula}),  and by the fact that $K$ is
    anti-blocking, we infer
	\begin{eqnarray*}
		V(K[j], -K[n-j])
        &=& {{n} \choose {j}}^{-1}\sum_{ E \in \Gnjc } \vol_j(P_{E} K) \cdot
        \vol_{n-j} (P_{E^\perp}(K))\\
		& = & {{n} \choose {j}}^{-1}\sum_{ E \in \Gnjc } \vol_j( K\cap E) \cdot \vol_{n-j} (P_{E^\perp}(K)),
	\end{eqnarray*}
	and by Lemma \ref{thm:rogers-shephard-sections-projections} we get  that 
	\begin{eqnarray}\label{eq:inequality-in-proof-of-godbersen}
	V(K[j], -K[n-j])
	&\le &  {\binom{n}{j}}^{-1}\sum_{ E \in \Gnjc } \binom{n}{j} \vol(K) \ = \ \binom{n}{j} \vol(K). 
	\end{eqnarray}
    Inspecting the equality condition, we see that $V(K[j], -K[n-j]) = \binom{n}{j} \vol(K)$ holds if and only if $
    \vol_{j}(K\cap E) \cdot \vol_{n-j}(P_{E^\perp}K) =
    \binom{n}{j}\vol_{n}(K)$ for all $j$-dimensional coordinate subspaces $E$.
    The equality condition in Lemma \ref{thm:rogers-shephard-sections-projections}, explained for anti-blocking bodies in Remark \ref{rem:equality-RS-antiB}, joint with Proposition~\ref{prop:RS-AB-equality} implies that this happens if and only
    if $K$ is a simplex.
\end{proof}

To complete this section, we would like  to further investigate  the quantity
\[
    \frac{\vol(K)}{V(K[j],-K[n-j])}   
\]
and show that it is a `good' measure of symmetry for anti-blocking bodies  in
the sense of Gr\"unbaum \cite{Grunbaum}. To this end we need to explain not
only that it is minimal for simplices, but that it is at most $1$, and that is
it maximal, namely equal to $1$, only if the body is a translate of a
centrally-symmetric body. While the former is covered by
Theorem~\ref{thm:god-for-closed-down}, the latter is immediate from the
equality condition in one of the simpler instances of the Alexandrov-Fenchel
inequality (see \cite[Page 398]{Schneider-book-NEW} and also \cite[Theorem
3.1]{Alesker-Dar-Milman}, where the equality case is discussed in Remark 1,
page 207), namely 
\[ 
    \prod_{i=1}^n \vol(K_i)^{1/n} \ \le \ V(K_1, \ldots, K_n),
\] 
with equality if and only if the bodies $K_i$ (assumed to be full dimensional) are pair-wise homothetic. 

Within the class of anti-blocking bodies, however, the variety of translates
of centrally-symmetric bodies is quite limited, and in Proposition
\ref{thm:cube-is-only-sym-AB} we show that in the case of anti-blocking
bodies, we get equality only for axes-parallel boxes, as these are the only
centrally-symmetric (up to translation) anti-blocking bodies.

\begin{prop}\label{thm:cube-is-only-sym-AB}
    The only anti-blocking bodies which have a center of symmetry are
    axes-parallel boxes, i.e., for $K\subseteq\RR_+^n$ anti-blocking,
	\[
        \exists x\in K\ \text{s.t. }K-x=x-K\ \iff\ 
        K = [0,a_1] \times [0,a_2] \times \cdots \times [0,a_n]
    \]
    for some $a_1,\dots,a_n \ge 0$.
\end{prop}

\begin{proof}
    The right-hand side of the equivalence has a center of symmetry and we
    only have to show the other direction.  If $K$ has a center of symmetry,
    then $K = 2x - K$.  Since $K \subseteq \RRn^n$, it follows that $K
    \subseteq \RRn^n \cap (2x - \RRn^n) = \prod_{i=1}^n [0,2x_i] =: B$. On the
    other hand, $2x \in K$ and the anti-blocking property implies that $B
    \subseteq K$. Hence $K = B$ is an axes-parallel box.
\end{proof}

\begin{rem}\rm 
	We mention that in our case, where we compute the mixed volume of two anti-blocking bodies ($K$ and $-K$), we do not need the full strength of Alexandrov's inequality and its equality case, and we can use the following simple argument, which illustrates once more that the main attribute of convex anti-blocking bodies, namely that 
    \begin{equation}\label{eqn:crucial}
    (K\cap E) \vee (K \cap E^\perp) \ \subseteq \  K \ \subseteq \  
    (K\cap E) \times (K\cap E^\perp)
\end{equation} 
is the a crucial concept we use. Indeed, as $K   \subseteq   P_EK \times
P_{E^\perp}K$, we may use Corollary \ref{cor:decomposition-conv}
Equation~\eqref{eq:jmixformula} which yields 
\begin{equation}\label{eqn:lower}   
    V(K[j], -K [n-j]) =  {\binom{n}{j}}^{-1}\sum_{E \in \Gnjc}
        \vol_j(P_E K) \cdot \vol_{n-j} (P_{E^\perp}K ) \ge \vol(K)
\end{equation}
            with equality if and only if $K = P_EK \times P_{E^\perp}K$ for all
    coordinate subspaces $E$. In a statement dual to
    Proposition~\ref{prop:RS-AB-equality}, this can happen only if $K$ is an
    axes-parallel box. More formally, in this case for any $E\in \Gnc$ we have 
    \[
        AK  = A(P_EK \times P_{E^\perp}K)  =  
        A(P_EK) \vee A(P_{E^\perp}K)  =
        P_E(AK) \vee P_{E^\perp}(AK)  = 
        (AK \cap E) \vee (E^\perp \cap AK) \, .
    \]
    By Proposition~\ref{prop:RS-AB-equality}, this is the case if and only if
    $AK$ is a simplex which is equivalent to $K$ being an axes-parallel box.
\end{rem}

\section{Near Mahler's conjecture}\label{sec:arnoundMahler}
Mahler's conjecture for centrally-symmetric bodies states that for any
centrally-symmetric convex body $K \subset \RR^n$ one has
$\vol(K)\vol(K^\circ) \ge 4^n /n!$ and with equality attained for the cube,
the cross-polytope, and also for all Hanner polytopes \cite{Hanner}. The
inequality has been proven for 1-unconditional bodies and is known as
Saint-Raymond's inequality. We prove the inequality for locally anti-blocking
bodies, along with related inequalities, in Subsection
\ref{subsec:Saint-Raymond}. In  Subsection \ref{subsec:C-bodies} we discuss
another special class of convex bodies, which includes the cube and the cross
polytope, for which we provide a bound on the Mahler volume which is close to
the conjectured bound.

\subsection{A mixed version of Saint-Raymond's inequality}
\label{subsec:Saint-Raymond}

For the proof of Mahler's conjecture for $1$-unconditional bodies,
Saint-Raymond proved the following beautiful inequality for anti-blocking
bodies. In this section, we prove an extension of this inequality to mixed
volumes.

The equality cases where determined by Meyer~\cite{Meyer-unconditional} and
Reisner~\cite{Reisner-unconditional}. An anti-blocking polytope $P \subset
\R^n$ is a \Def{reduced Hanner polytope} if $P = [0,1] \subseteq\R$ or there is
a coordinate subspace $E$ such that $P \cap E$ and $P \cap E^\perp$ are
reduced Hanner and
\[
    P \ = \ (P \cap E) \times (P \cap E^\perp) \quad \text{ or } \quad
    P \ = \ (P \cap E) \vee (P \cap E^\perp) \, .
\]
Equivalently, we could say that the unconditional body $\bigcup_\sigma \sigma P$ is a Hanner polytope; see \cite[Section 10.7]{Schneider-book-NEW}.

\begin{thm}[Saint-Raymond~\cite{Saint-Raymond}] \label{thm:saint-raymond-ab}
	Let $K \subseteq\RRn^n$ be a full-dimensional anti-blocking body. Then
	\[
	    \vol(K) \cdot \vol(A(K)) \ \ge \ \frac{1}{n!} \, .
	\]
    Equality is attained precisely for the reduced Hanner polytopes.
\end{thm}

The proof of Mahler's conjecture for a $1$-unconditional body $K \subseteq\R^n$
now follows from the fact that $\vol(K) = 2^n \vol(K \cap \RRn^n)$ and $K\dual
\cap \RRn^n = A(K \cap \RRn^n)$. With only a little more effort, we can show
that locally anti-blocking bodies also satisfy Mahler's conjecture. The
functional version of the following result was observed by 
Fradelizi--Meyer~\cite[Remark on p.~1451]{FradeliziMeyer}.

\begin{cor}\label{cor:mahler-for-locally-ab}
    Let $K \subseteq\RR^n$ be a locally anti-blocking body with $0$ in the
    interior. Then
	\[
	    \vol(K) \cdot \vol(K^\circ) \ \ge \ \frac{4^n}{n!} \, .
	\]
    Equality holds precisely when each $K_\sigma$ is a reduced Hanner
    polytope.
\end{cor}
\begin{proof}
    We recall from Lemma~\ref{lem:polarity-AB-decomposes} that 
    $(K\dual )_\sigma  = A(K_\sigma)$ and compute
	\begin{align*}
	\vol(K)
	\vol(K\dual) \ &= \ 
	\sum_{\sigma} \vol(K_\sigma)
	\sum_{\sigma} \vol(A(K_\sigma)) \\
	\ &\ge \
	\left(\sum_{\sigma} \sqrt{\vol(K_\sigma)\vol(A(K)_\sigma)} \right)^2\\
	\ &\ge \
	\left(\sum_{\sigma} \frac{1}{\sqrt{n!}} \right)^2 
	\ = \ \frac{4^n}{n!}\, , 
	\end{align*}
    where the first inequality is Cauchy--Schwartz and the second is
    Theorem~\ref{thm:saint-raymond-ab}.
\end{proof}

Using the same tools, we proceed in proving the main result of this
subsection, Theorem \ref{thm:mahler-for-vj}, which is a version of
Saint-Raymond's inequality for mixed volumes.

\begin{proof}[Proof of Theorem \ref{thm:mahler-for-vj}] 
    We begin by proving~\eqref{eq:mixed-sr-simple}. Let $E\in \Gnjc$. Then
    $P_EK \subseteq E \cong \RR_+^j$ is an anti-blocking body with dual
    $A_EP_EK=P_EAK$, $A_E$ is the duality operator for anti-blocking bodies in
    $E$.  Applying Saint-Raymond's Theorem~\ref{thm:saint-raymond-ab} we get
    $\vol(P_EK)\vol(P_EAK)  \ge  1/j!,$ and hence 
	\begin{equation}\label{eq:product-of-projections-ineq}
        \vol(P_EK)\vol(P_{E^\perp}T) \vol(P_EAK)\vol(P_{E^\perp}AT) \ \ge \
        \frac{1}{j! (n-j)!}. 
	\end{equation}
    Using the Cauchy-Schwarz inequality and the decomposition of
    Lemma~\ref{lem:decomposition-sum}, we get 
	\begin{align*}
	& V(K[j],-T[n-j])V(AK[j],-AT[n-j])\\
    &=\left(\binom{n}{j}^{-1} \sum_{ E\in \Gnjc }
    \vol(P_{E}K)\vol(P_{E^\perp}T) \right)\cdot  \left(\binom{n}{j}^{-1}
    \sum_{E\in \Gnjc} \vol(P_{E}AK)\vol(P_{E^\perp}AT) \right) \\
	&\ge  \binom{n}{j}^{-2}\left(   \sum_{ E\in \Gnjc} 
	\sqrt{\vol(P_{E}K)\vol(P_{E^\perp}T) \vol(P_{E}AK)\vol(P_{E^\perp}AT)}
    \right)^2 \\
    &\ge \binom{n}{j}^{-2}\left(\binom{n}{j}
    \frac{1}{\sqrt{j!}\sqrt{(n-j)!}}\right)^2 = \frac{1}{j!}\frac{1}{(n-j)!},
	\end{align*}
    where the last inequality is due to
    \eqref{eq:product-of-projections-ineq}.
	
    For the general case, \eqref{eq:mixed-sr-comlicated}, we apply the repeated
    Alexandrov-Fenchel inequality (see~\cite[Eqn.~(7.64)]{Schneider-book-NEW})
	\[
        V(K_1,\dots, K_n)^m \ \ge \ \prod_{i=1}^m V(K_i[m],K_{m+1},\dots, K_n)
    \]
	and \eqref{eq:mixed-sr-simple} to get
	\begin{align*}
	&V(K_1,\dots, K_j, -T_1,\dots -T_{n-j})^{j(n-j)} V(AK_1,\dots, AK_j, -AT_1,\dots -AT_{n-j})^{j(n-j)}\\
	&\ge \prod_{i=1}^j V(K_i[j],-T_1,\dots, -T_{n-j})^{n-j}V(AK_i[j],-AT_1,\dots, -AT_{n-j})^{n-j} \\
    &\ge \prod_{i=1}^j\prod_{k=1}^{n-j}
    V(K_i[j],-T_k[n-j])V(AK_i[j],-AT_k[n-j]) \ge
    \left(\frac{1}{j!(n-j)!} \right)^{j(n-j)}.
	\end{align*}
	Taking  $j(n-j)$-th root on both sides we get the desired result.
\end{proof}

Plugging in $K=T$ in Theorem~\ref{thm:mahler-for-vj} results in a generalized
version of Saint-Raymond's Theorem \ref{thm:saint-raymond-ab}.
\begin{cor}\label{cor:mahler-for-vj(K)}
	Let $K\subseteq\R^n_+$ be a full-dimensional anti-blocking body and $0 \le
    j \le n$. Then 
	\[
	    V(K[j],-K[n-j])V(AK[j],-AK[n-j]) \ \ge \ \frac{1}{j!(n-j)!}.
	\]
\end{cor}

\begin{rem}
    Inspecting the proofs above, we see that the equality case requires $P_EK$
    and  $P_{E^\perp}K$ to be reduced Hanner polytopes for every $E \in
    \Gnjc$. This, however, does not imply that $K$ is a reduced Hanner
    polytope. In fact, every $2$-dimensional anti-blocking polytope has this
    property for $j = 1$.
\end{rem}

\subsection{Nearly Mahler for $C$-bodies}\label{subsec:C-bodies}

Recall from the introduction that the $C$-body associated to convex
bodies $K,T \subseteq\R^n$ is defined by
\[
    C(K,T)  \ := \ (K \times \{1\}) \vee (T \times \{-1\}) \subseteq \RR^{n+1} .
\]
The $C$-body construction was introduced by Rogers and Shephard
in~\cite{RS-sections-projections}. These bodies were used in~\cite{Artstein-Note,
ArtsteinKeshet} in relation to Godbersen's conjecture and
in~\cite{Sanyal-ABbodies} in connection with combinatorics of partially
ordered sets. In discrete and convex geometry, the construction of
$C$-bodies is referred to as the \emph{Cayley construction}, which captures
Minkowski sums of $K$ and $T$. For $-1 \le
\lambda \le 1$ 
\begin{equation}\label{eqn:Cayley}
    C(K,T) \cap (\R^n \times \{\lambda\}) \ = \ ( \tfrac{1+\lambda}{2} K +
    \tfrac{1-\lambda}{2} T) \times \{\lambda\} \, .
\end{equation}
The \emph{Cayley trick} states if $K + T = \bigcup_i K_i + T_i$  is a
dissection with $K_i \subseteq K$ and $T_i \subseteq T$ then
\[
    C(K,T) \ = \ \bigcup_i C(K_i,T_i)
\]
is a dissection of $C(K,T)$; cf.~\cite{DLRS} for details. 

 It is useful to note that the volume of a $C$-body can be written in terms of the mixed volumes of its generating sets, a fact used already in \cite{RS-difference}, the proof of which is included for completeness. 
\begin{lem}\label{lem:volofC}%
    \label{lem:vols-of-vonstructions}
    Let $K,T \subseteq\RRn^n$ be convex bodies. Then
    \[
        \vol(C(K,T)) \ = \ \frac{2}{n+1} \sum_{j=0}^n V(K[j],T[n-j]) \, .
    \]
\end{lem}

\begin{proof} The proof reduces to a computation of an integral: 
	\begin{eqnarray*}
	 \vol_{n+1}(C(K,T)) & = & \int_{-1}^1 \vol_n(C(K,T) \cap \RR^n \times \{\lambda\})d\lambda \\
	 & = & \int_{-1}^1 \vol_n ( \tfrac{1+\lambda}{2} K -
	 \tfrac{1-\lambda}{2} T) d\lambda  \\
	 &=& \int_{-1}^1  \sum_{j=0}^n \binom{n}{j} \left(
     \tfrac{1+\lambda}{2}\right)^j \left( \tfrac{1-\lambda}{2} \right)^{n-j}
     V(K[j], T[n-j])d\lambda \\ &=& \frac{2}{n+1} \sum_{j=0}^n V(K[j],T[n-j]).
     \qedhere
	\end{eqnarray*}
\end{proof}

It was  shown in \cite{Sanyal-ABbodies} that polarity respects the structure
of $C$-bodies.

\begin{thm}[See {\cite[Thm.~3.4]{Sanyal-ABbodies}}]
    \label{thm:duality-respects-C}
    Let $K, T \subseteq\RRn^n$ be full-dimensional anti-blocking bodies. Then
	\[ 
        C(K,-T)^\circ \ = \  C(-2AT, 2AK) \, .
    \]
\end{thm}

We will be mostly interested in the case $C(K,-K)$, which we  denoted by
$C(K)$. Obviously $C(K)$ is centrally-symmetric and Mahler's conjecture
should hold. Note that both the cube and the cross-polytope are special
cases of this type, with the base $K$ chosen to be a cube and a simplex, respectively. 

In view of Theorem~\ref{thm:duality-respects-C},
Lemma~\ref{lem:vols-of-vonstructions}, and
Corollary~\ref{cor:decomposition-conv}, Mahler's conjecture is equivalent to 
\begin{equation}\label{eq:mahler-for-C(K)}
    \vol(K \vee -K) \vol( AK \vee -AK ) 
    \ \ge \
    2^n\frac{n+1}{n!} \, ,
\end{equation}
which is why our first goal is to prove Proposition
\ref{prop:nearly-mahler-1}.

\begin{proof}[Proof of Proposition~\ref{prop:nearly-mahler-1}]
	Using Cauchy-Schwarz together with the bounds of Theorem \ref{thm:mahler-for-vj} we have
    \begin{align*} 
	\vol(K \vee -K)\vol(AK \vee -AK) &\ = \ \Bigl (  \sum_{j=0}^n V(K[j],-K[n-j])\Bigr )  \Bigl(  \sum_{j=0}^n
    V(AK[j],-AK[n-j]) \Bigr) \\
    &\ \ge \ \Bigl(   \sum_{j=0}^n \sqrt{V(K[j],-K[n-j]) V(AK[j],-AK[n-j])
    } \Bigr)^2 	\\
	& \ \ge \ \frac{1}{n!}  \left(   \sum_{j=0}^n \binom{n}{j}^{1/2}
    \right)^2. \qedhere
	\end{align*} 
\end{proof}
\begin{rem}
    The above bound is clearly not tight, as
    \[
        \left(\sum_{j=0}^n \binom{n}{j}^{1/2}\right)^2 \ < \ 2^n (n+1),
    \] 
    e.g. by Cauchy-Schwartz.
\end{rem}

Using another, perhaps simpler, route, we can show a slightly weaker bound:
\begin{prop} \label{prop:nearly-mahler-2}
    Let $K\subseteq\RR_{+}^n$ be a full-dimensional anti-blocking body. Then
	\[
        \vol(K \vee -K)\vol(AK \vee -AK) \ \ge \  \frac{2^n}{n!}\sqrt{\pi
        n/2} \, .  
	\]
\end{prop}

\begin{proof}
	Consider the locally anti-blocking body $K-K\subseteq\RR^n$ and its polar
    $(K-K)^\circ = AK \vee -AK$. We know by
    Lemma~\ref{lem:vols-of-vonstructions} that 
    \begin{align*} 
            \vol(K-K) \ &= \ \sum_{j=0}^n \binom{n}{j} V(K[j],-K[n-j]) \quad
            \text{and}\\
	       \vol (AK \vee -AK) \ &= \ \sum_{j=0}^n  V(AK[j],-AK[n-j]) \, .
    \end{align*}
	From Corollary~\ref{cor:mahler-for-locally-ab} we get that 
    \[
        \Bigl( \sum_{j=0}^n \binom{n}{j} V(K[j],-K[n-j]) \Bigr) 
        \Bigl( \sum_{j=0}^n V(AK[j],-AK[n-j]) \Bigr) \ \ge \ \frac{4^n}{n!}.
	\]
	We may use the crude bound $\binom{n}{j}\le \binom{n}{[n/2]}$ to get 
	\[
	\Bigl(\sum_{j=0}^n V(K[j],-K[n-j]) \Bigr) \Bigl(\sum_{j=0}^n
    V(AK[j],-AK[n-j]) \Bigr)
	\ge  \frac{2^n}{\binom{n}{[n/2]}}\frac{2^n}{n!} \ge \sqrt{\pi n/2} \frac{2^n}{n!}
	\]
	where we bound the central binomial coefficient using standard bounds, see e.g. \cite{Cameron1994combinatorics}.
\end{proof}

Summarizing, we have proved Proposition \ref{prop:nearly-mahler-1}, which
paired with Lemma \ref{lem:volofC} and Theorem \ref{thm:duality-respects-C}
gives the desired bound for the volume product of $C(K)$ and $C(K)^\circ$.
Before moving to the proof, let us discuss the bodies $C_\lambda(K)$.

For two convex bodies $K, T \subseteq\R^n$ one may construct a family of convex
bodies
\[
    C_\lambda(K,-T) \ := \ \conv( K \times \{\lambda\} \cup -T \times
    \{-\lambda\} \cup [-\mathbf{e}_{n+1}, \mathbf{e}_{n+1} ] ) \subset
    \RR^{n+1} \, ,
\]
for $-1 \le \lambda \le 1$.

This family is in fact a ``system of moving shadows'' in direction $e_{n+1}$,
as defined and studied in~\cite{Shephard-ShadowSys} by Shephard. Let $K
\subseteq\R^n$ be a convex body and $\theta \in \R^n$.  A {\Def{shadow system}}
in direction $\theta$ is the one-parameter family of convex bodies defined by
$K_t := \conv \{ x+ \alpha(x)t\theta: x\in K\}$, where $\alpha$ is a bounded
function on $K$. 

\begin{thm}[Shephard~\cite{Shephard-ShadowSys}]\label{thm:shadow_vol}
    Let $K^{(i)} \subseteq\R^n$ be convex bodies for $i=1,\dots,n$. Then the
    function $t \mapsto V(K_t^{(1)}, \ldots, K_t^{(n)})$ is convex.
\end{thm}

\begin{cor}\label{cor:vol_Clambda}
    Let $K,T \subseteq\R^n$ be two anti-blocking bodies. Then for all $-1 \le
    \lambda \le 1$
    \[
        \vol( C_\lambda(K,-T)) \ = \ \vol(C(K,-T)) \, .
    \]
\end{cor}
\begin{proof}
    Since $C_\lambda(K,-T)$ is a shadow system, the claim follows from
    Theorem~\ref{thm:shadow_vol} together with the fact that 
    $\vol(C_1(K,-T))=\vol(C_{0}(K,-T))=\vol(C_{-1}(K,-T))$.
\end{proof}

\begin{proof}[Proof of Theorem \ref{thm:mahler-for-c-bodies}]
    It is shown in~\cite{Campi-Gronchi2006volume}, that $1/\vol(K_\lambda)$ is
    a convex function of $\lambda$ when $K_\lambda$ is a shadow system. Hence,
    together with Corollary~\ref{cor:vol_Clambda} we have 
    \[
        \vol(C_\lambda(K,-T))\vol(C_\lambda(K,-T)^\circ)  \ \ge \
        \vol(C_1(K,-T))\vol(C_1(K,-T)^\circ)
    \]

    Thus, it is enough to prove the lower bound for $\lambda = 1$, namely for
    the body $C(K)$ itself. We see	that 
    \begin{eqnarray*} \vol(C(K))\vol(C(K)^\circ) &=& 2^{n} \vol(C(K)) \vol(C(AK)) \\&=& \frac{2^{n+2}}{(n+1)^2} \left(\sum_{j=0}^n V(K[j], -K[n-j])\right) 
\left(\sum_{j=0}^n V(AK[j], -AK[n-j])\right). 
\end{eqnarray*}
 where the first equality is by Theorem \ref{thm:duality-respects-C} and the second by Lemma \ref{lem:volofC}. 
Using Corollary \ref{cor:decomposition-conv} we may rewrite this as 
\[ \vol(C(K))\vol(C(K)^\circ) =   \frac{2^{n+2}}{(n+1)^2} \vol(K\vee -K) \vol(AK \vee -AK). 
\] 
From Proposition \ref{prop:nearly-mahler-1} we obtain the lower bound
\[  \vol(C(K))\vol(C(K)^\circ)\ge    \frac{2^{n+2}}{(n+1)^2}  \frac{1}{n!}  \left(   \sum_{j=0}^n \binom{n}{j}^{1/2} \right)^2. \] 
A trivial lower bound on the right hand side is $\frac{1}{n+1}\cdot 4^{n+1}/(n+1)!$. 
With some more work one may use approximations of binomial coefficients to arrive at the bound  
\[
\vol(C(K))\vol(C(K)^\circ) \ge \frac{\sqrt{2\pi n}}{n+1} \cdot
\frac{4^{n+1}}{(n+1)!}.
\]
Alternatively, one may use Proposition \ref{prop:nearly-mahler-2} instead, and get a bound of the same order in $n$, with slightly worse numerical constants. 
\end{proof}

\section{A lower bound on mixed volumes}
\label{sec:Kleitman}

In this section we will prove Theorem \ref{thm:kleitman-bound-on-mixed}, which is in spirit a bound on mixed volumes of anti-blocking bodies in the opposite direction to the bound in Godbersen's conjecture. 
As we explain in Section \ref{sec:God}, the lower bound 
 $\vol(K) = V(K,\dots ,K) \leq V(K[j], -K[n-j])$ is true for any convex body $K$. When $-K$ is replaced by $-T$, the same method produces the lower bound $\vol(K)^{j/n}\vol(T)^{(n-j)/n} \leq V(K[j], -T[n-j])$. It turns out that when $K$ and $T$ are anti-blocking bodies, a much stronger lower bound is available, and we will show, as claimed in Theorem \ref{thm:kleitman-bound-on-mixed}, that in such a case 
\[  V(K[j], T[n-j])  \leq V(K[j], -T[n-j]).\]

With this in hand, one gets, using the decomposition of mixed volumes, the second part of Theorem \ref{thm:kleitman-bound-on-mixed}, namely

\begin{thm}\label{thm:bound-of-vol(K+T)}
    For two  anti-blocking bodies $K, T \subseteq\RR^n$ it holds that 
	\begin{equation}\label{eq:bound-of-vol(K+T)}
	    \vol(K + T) \ \leq \ \vol(K - T) .
	\end{equation}
\end{thm}

 Theorem \ref{thm:bound-of-vol(K+T)}
  turns out to have an independent simple proof relying on the Reverse Kleitman inequality, which we first describe.  
The following theorem of Bollob\'as, Leader and
Radcliffe was given in \cite{Bollobas-Reverse-Kleitman}.
\begin{thm}[Reverse Kleitman inequality for order-convex bodies]
    \label{thm:reverse-kleitman}
	Given an order-convex body $L \subseteq\RR^n$, 
	\[ 
        \vol(\Delta(L)) \ \leq \ \vol(L) \, ,
    \]
	where $\Delta(L) = (L-L)\cap \RRn^n$.
\end{thm}
\begin{proof}[Proof of Theorem \ref{thm:bound-of-vol(K+T)}]
    Given two anti-blocking bodies $K, T \subseteq\RRn^n$, we set $L =
    K - T$, which is locally anti-blocking. Applying
    Lemma~\ref{lem:decomposition-sum}, we get that $\vol(L) = \sum_{E\in \Gnc}
    \vol(P_{E}K) \vol(P_{E^\perp}T)$. Again by Lemma
    \ref{lem:locally-ab-sum-decomp}
	\[
        L-L \ = \  \bigcup_{\sigma} L_\sigma + (-L)_\sigma .
    \]
    Since As $L_{(+,\cdots,+)} = K$ and $(-L)_{(+,\cdots, +)}= L_{(-,\cdots,
    -)} = T$, the former decomposition gives $\Delta(L) = K+T$.
    Since $L$ is order-convex, we may apply 
    Theorem~\ref{thm:reverse-kleitman} to conclude
	\[
        \vol(K + T) \ = \ \Delta(L) \ \leq \ \vol(L) \ = \ \sum_{E \in \Gnc}
        \vol(P_{E}K) \vol(P_{E^\perp} T) \ = \ \vol(K - T)  \, .
        \qedhere
	\]
\end{proof}

Theorem \ref{thm:bound-of-vol(K+T)} can be rewritten as
\[
    \sum_{j=0}^n \binom{n}{j} V(K[j],T[n-j]) \ \leq \ \sum_{j=0}^n
    \binom{n}{j}  V(K[j],-T[n-j]),
\]
and with this in hand it is natural to inquire whether the term-by-term inequality holds.
Theorem~\ref{thm:kleitman-bound-on-mixed} claims this to be true  (and will also provide a second proof for Theorem \ref{thm:bound-of-vol(K+T)}) and we
proceed with a lemma necessary for its proof.

\begin{lem}\label{lem:hyp_Steiner}
	Let $K,T \subseteq\R^n$ be two convex bodies and let $H$ be a hyperplane
	such that $P_H K = K \cap H$ and $P_H T = T \cap H$. Let $S_H K, S_H T$ be
	the Steiner symmetral of $K$ and $T$ with respect to $H$. Then
	\[
	V(S_HK[j],S_HT[n-j]) \ \le \ V(K[j],T[n-j]) \, .
	\]
\end{lem}
\begin{proof}
	We may assume that $H = \{ x : x_n = 0\}$. Let $K' := P_H K$ and for $x
	\in K'$, let $\alpha_K(x) := \max \{ a \ge 0 : (x,a) \in K \}$. For $0
	\le t \le 1$ define
	\[
	S^t_H K  := \{ x - \alpha_K(x) t e_n : x \in K \}  \ = \ \{ (x,x_n) \in
	K' \times \RR : -t \alpha_K(x) \le x_n \le (1-t) \alpha_K(x) \} \, .
	\]
	This is a convex body and $S^t_H K$ is a shadow system. Note that $S_H K =
	S^{1/2}_H K$ and $S^1_H K$ is the reflection of $K$ in the hyperplane $H$.
	Applied to $T$, this yields a shadow system $(S_H^t K[j], S_H^t T[n-j])$
	and by Theorem~\ref{thm:shadow_vol}, we get
	\begin{align*}
	V(S_HK[j],S_HT[n-j]) \ &= \ 
	V(S^{1/2}_HK[j],S^{1/2}_HT[n-j]) \\ \ &\le \ 
	\tfrac{1}{2} V(S^{0}_HK[j],S^{0}_HT[n-j]) + 
	\tfrac{1}{2} V(S^{1}_HK[j],S^{1}_HT[n-j]) \\
	\ &= \ 
	V(K[j],T[n-j]) \, . \qedhere
	\end{align*}
\end{proof}

\begin{proof}[Proof of Theorem \ref{thm:kleitman-bound-on-mixed}]
    For an anti-blocking body $K \subseteq\RRn^n$, denote the associated
    $1$-unconditional body by
	\[
        \widehat{K} \ = \  \bigcup_{\sigma \in \PM^n} \sigma K \, .
    \]
    By Lemma~\ref{lem:locally-ab-sum-decomp}, we have that
	\[
        \vol(\widehat{K} + \widehat{T}) \ = \ 
        \sum_{\sigma} \vol(K_\sigma + T_\sigma) \ = \ 
        \sum_{\sigma} \vol(K + T)
    \]
    and thus
	\[
        V(K[j],T[n-j]) \ = \ 
        \frac{1}{2^n}V(\widehat{K}[j],\widehat{T}[n-j]) \ = \
        V(\tfrac{1}{2}\widehat{K}[j],\tfrac{1}{2}\widehat{T}[n-j]).
    \]
	Hence, it will suffice to show that 
	\begin{equation}\label{eq:V-j-of-locally-ab}
        V(\tfrac{1}{2}\widehat{K}[j],\tfrac{1}{2}\widehat{T}[n-j]) 
        \ \leq \
        V(K[j],-T[n-j]) \, .
	\end{equation}

    Let $H_i = \{ x \in \R^n : x_i = 0 \}$ for $i=1,\dots,n$. Then $K$ and $T$
    have the property that $P_{H_i}K = K \cap H_i$ for all $i=1,\dots,n$.
    Hence, we can repeatedly apply Lemma~\ref{lem:hyp_Steiner} to get
    \[
        V(S_{H_1}S_{H_2}\dots S_{H_n}K[j],
        S_{H_1}S_{H_2}\dots S_{H_n}T[n-j]) \ \le  \
        V(K[j],-T[n-j]) \, .
    \]
    Note that $\frac{1}{2} K \subseteq S_{H_1}S_{H_2}\dots S_{H_n}K$ and since
    $S_{H_1}S_{H_2}\dots S_{H_n}K$ is symmetric with respect to reflections in
    all coordinate hyperplanes $H_i$, we get that $\tfrac{1}{2}\widehat{K}
    \subseteq S_{H_1}S_{H_2}\dots S_{H_n}K$ and, in fact, equality. Appealing
    to the monotonicity of mixed volumes then completes the proof.
\end{proof}

\section{Applications to partially ordered sets}\label{sec:posets}
\newcommand\chain{\mathcal{C}}%
In this section, we collect some of the combinatorial consequences
derived from our geometric constructions. 

\newcommand\SymGrp{\mathfrak{S}}%
Recall that a finite \Def{poset} is a pair $P = (V,\preceq)$ where $V$ is
finite and $\preceq$ is a reflexive, anti-symmetric, and transitive binary
relation on $V$. A \Def{linear extension} is a bijection $\ell : V \to [n] :=
\{1,\dots,n\}$ with $n = |V|$ so that $a \prec b$ implies $\ell(a) < \ell(b)$
for all $a,b \in V$. The number $e(P)$ of linear extensions of $P$ is an
important enumerative invariant in the combinatorial theory of posets. 

\newcommand\cG{\overline{G}}%
It turns out~\cite{Sidorenko,TwoPoset} that the number of linear extensions of
$P$ only depends on its \Def{comparability graph} $G(P)$, that is, the
undirected graph on nodes $V$ with edges $uv$ for $u \prec v$ or $v \prec u$.
Comparability graphs as well as \Def{co-comparability graphs} $\cG(P) = (V, \{
uv : u \not\preceq v, v \not\preceq u\})$ are well-studied. Co-comparability
graphs belong to the class of string graphs and they contain all interval
graphs. Those posets whose co-comparability graphs are also comparability
graphs are characterized as follows: Let $\SymGrp_n$ be the permutations on
$[n]$. For   $\pi \in \SymGrp_n$ define the poset $P_\pi =
([n],\preceq_\pi)$ with 
\[
    a \prec_\pi b \quad :\Longleftrightarrow \quad a < b \text{ and } \pi_a <
    \pi_b \, .
\]
\newcommand\opi{{\overline{\pi}}}%
It is easy to see that $\cG(P_\pi) = G(P_{\opi})$ where $\opi := n+1 - \pi_a$
The posets $P_\pi$ are called \Def{two-dimensional} and Trotter~\cite{Trotter}
showed $\cG(P)$ is a comparability graph if and only if $P$ is isomorphic to a
two-dimensional poset.

Sidorenko~\cite{Sidorenko} studied the relation between $e(P_\pi)$ and
$e(P_{\opi})$ and proved the following inequality. To describe the theorem with its equality cases we need one more definition.
Suppose $P$ and $Q$ are posets on disjoint ground sets. Then $P \cup Q$ is
given a partial order by either making elements from $P$ and $Q$ incomparable
or setting $p \prec q$ for all $p \in P$ and $q \in Q$. The former is called
the \Def{parallel composition} while the latter is called the \Def{series
composition}. A poset is called \Def{series-parallel} if it has a single
elements or if it is the series or parallel composition of series-parallel
posets.

\begin{thm}[Sidorenko~\cite{Sidorenko}]\label{thm:sidorenko}
    For any $\pi \in \SymGrp_n$ 
    \[
        e(P_\pi) \cdot e(P_\opi) \ \ge \  n! \, .
    \]
    Equality is attained precisely when $P_\pi$ is series-parallel.
\end{thm}

\newcommand\Prob{\mathbb{P}}%
Bollob\'as, Brightwell, Sidorenko~\cite{BBS}, in addition to giving a simple
geometric proof of Theorem~\ref{thm:sidorenko} (see below), gave the following
interpretation: Consider the uniform probability distribution on $\SymGrp_n$.
For $\pi \in \SymGrp_n$, let $X_\pi$ be the random variable that is $=1$ for
linear extensions of $P_\pi$.  That is, $\Prob(X_\pi) = \frac{e(P_\pi)}{n!}$.
The identity permutation is the unique permutation that is simultaneously a linear
extension of $P_\pi$ and $P_\opi$, which shows $\Prob(X_\pi \wedge X_\opi) =
\frac{1}{n!}$.  Theorem~\ref{thm:sidorenko} then states that the pair $X_\pi,
X_\opi$ is \emph{negatively correlated}:
\[
    \Prob(X_\pi \wedge X_\opi) \ \le \ \Prob(X_\pi) \cdot \Prob(X_\opi) \, .
\]
In the following, we will generalize Sidorenko's inequality.

For $0 \le j \le n$ and a permutation $\tau \in \SymGrp_n$ set $J :=
\tau^{-1}([j])$. We define $L_j\tau : J \to \{1,\dots,j\}$ to be the restriction
of $\tau$ to $J$ and $R_j\tau : J^c \to \{1,\dots,n-j\}$ by $R_j\tau(a) :=
\tau(a) - j$. A \Def{double poset} is a pair of posets $(P,Q)$ on the same
ground set $[n]$.  Double posets were introduced by Malvenuto and
Reutenauer~\cite{MR} and studied from a geometric point of view
in~\cite{Sanyal-ABbodies}. For a subset $J \subseteq [n]$ we denote by $P|_J$
the restriction of the partial order $\preceq_P$ to $J$.

Now, for a double poset $(P,Q)$ and $0 \le j \le n$, we define $e_j(P,Q)$ to
be the number of permutations $\tau \in \SymGrp_n$ such that $L_j\tau$ is a linear
extension for $P|_J$ and $R_j\tau$ is a linear extension for $Q|_{J^c}$.

\newcommand\osig{{\overline{\sigma}}}%
\begin{thm}[Mixed Sidorenko inequality] \label{thm:gen_sidorenko}
    Let $\pi, \sigma \in \SymGrp_n$ and $0 \le j \le n$. Then
    \[
        e_j(P_\pi,P_\sigma) \cdot e_j(P_\opi,P_\osig) \ \ge \ n! \binom{n}{j} \, .
    \]
\end{thm}

For $j = 0$ this is exactly Theorem~\ref{thm:sidorenko}.

Before we give the proof, let us recast this into probabilistic terms. For
$\pi,\sigma \in \SymGrp_n$, let us write $X_{\pi,\sigma}$ for the random
variable that is one on all permutations $\tau$ such that $L_j\tau$ is a
linear extension for $P_\pi|_J$ and $R_j\tau$ is a linear extension for
$P_\sigma|_{J^c}$.

\begin{cor}
    For every $\pi,\sigma \in \SymGrp_n$, the pair $X_{\pi,\sigma},
    X_{\opi,\osig}$ is negatively correlated.
\end{cor}
\begin{proof}
    Observe that for fixed $0 \le j \le n$, there is a bijection between
    $\SymGrp_n$ and the set of triples $(J,L,R)$ where $J\subseteq [n]$ with
    $|J| = j$ and $L : J \to [j]$ and $R : J^c \to [n-j]$ are bijections.
    Hence
    \[
        e_j(P,Q) \ = \ \sum_J e(P|_J) \cdot e(Q|_{J^c}) \, ,
    \]
    where the sum is over all subsets $J \subseteq [n]$ of cardinality $j$. In
    particular
    \[
        \Prob(X_{\pi,\sigma}) \ = \ \binom{n}{j}^{-1} \sum_{J} 
        \frac{e(P_\pi|_J)}{j!}
        \frac{e(P_\sigma|_{J^c})}{(n-j)!} \ = \ \frac{e_j(P,Q)}{n!} \, .
    \]
    It is easy to see that the number of $\tau \in \SymGrp_n$ such that
    $(L_j\tau,R_j\tau)$ is a pair of linear extensions for $(P_\pi,P_\sigma)$
    and $(P_\opi,P_\osig)$ is $\binom{n}{j}$ and hence 
    \[
        \Prob( X_{\pi,\sigma} \wedge X_{\opi,\osig} ) \ = \
        \frac{\binom{n}{j}}{n!} \ = \ \frac{1}{j! (n-j)!} \, .
    \]
    The claim now follows from Theorem~\ref{thm:gen_sidorenko}.
\end{proof}

\newcommand\dposet{\mathbf{P}}%
\newcommand\Chain{\mathcal{C}}%
For the proof of Theorem~\ref{thm:gen_sidorenko}, recall from~\cite{TwoPoset}
that the \Def{chain polytope} $\Chain(P)$ associated to a poset $P =
([n],\preceq)$ is the anti-blocking polytope given by points $x \in \RRn^n$
with
\begin{equation}\label{eqn:chain}
    x_{a_1} + x_{a_2} + \cdots + x_{a_k} \ \le \ 1 \quad \text{ for all chains
    }
    a_1 \prec a_2 \prec \cdots \prec a_k \, .
\end{equation}
It was shown in~\cite{TwoPoset} that $\vol_n(\Chain(P)) = \frac{e(P)}{n!}$ and
used in~\cite{BBS} to prove Theorem~\ref{thm:sidorenko} as follows.  Let $G =
(V,E)$ be an undirected graph. A set $S \subseteq V$ is \Def{stable} if there
is no pair of nodes $u,v \in S$ with $uv \in E$. We write $1_S \in \{0,1\}^V$
for the characteristic vector of a set $S \subseteq V$. The \Def{stable set
polytope} of $G$ is the anti-blocking polytope
\newcommand\Stab{\mathcal{S}}%
\[
    \Stab_G \ := \ \conv( 1_S : S \subseteq V \text{ stable}) \, .
\]
It is not difficult to verify that $\Chain(P) = \Stab_{G(P)}$. Comparability
graphs are perfect graphs and for perfect graphs $G$ Lov\'{a}sz~\cite{Lovasz}
showed that the anti-blocking dual is
\[
    A\Stab_G \ = \ \Stab_{\overline{G}} \, ,
\]
where $\overline{G} = (V,\{ uv : uv \not \in E \})$ is the \Def{complement
graph} of $G$. Saint-Raymond's inequality~\eqref{eqn:SaintRaymond} now yields
\[
    \frac{e(P_\pi)}{n!} \cdot \frac{e(P_\opi)}{n!} \ = \ 
    \vol \Chain(P_\pi) \cdot 
    \vol A\Chain(P_\pi) \ \ge \ \frac{1}{n!} \, .
\]
The equality cases follow from the observation that $\Chain(P)$ is a
reduced Hanner polytope if and only if $P$ is series-parallel.

\begin{proof}[Proof of Theorem~\ref{thm:gen_sidorenko}]
    Let $P = ([n],\preceq)$ be a poset and $J \subseteq [n]$.  For the
    coordinate subspace $E = \{ x \in \RR^n : x_i = 0 \text{ for } i \not\in J
    \}$, it follows from~\eqref{eqn:chain} that 
    \[
        P_E \Chain(P) \ = \ \Chain(P) \cap E \ = \ \Chain(P|_J) \, .
    \]
    Let $(P,Q)$ be a double poset. From Lemma~\ref{lem:decomposition-sum}, we
    get that for $0 \le j
    \le n$
    \begin{align*}
    V(\chain(P)[j],-\chain(Q)[n-j]) 
        \ &= \ \binom{n}{j}^{-1} \sum_{|J|=j} \Vol \chain(P|_J) \Vol
    \chain(Q|_{J^c}) \\
    \ &= \ \binom{n}{j}^{-1} \sum_{|J|=j} \frac{e(P|_J)}{j!}
    \frac{e(Q|_{J^c})}{(n-j)!} \ = \ \frac{e_j(P,Q)}{n!} \, .
    \end{align*}
    Theorem~\ref{thm:mahler-for-vj}\eqref{eq:mixed-sr-simple} now yields
    \[
    \frac{e_j(P_\pi,P_\sigma)}{n!} \cdot \frac{e_j(P_\opi,P_\osig)}{n!}  =  
        V(\Chain(P_\pi)[j], \Chain(P_\sigma)[n-j]) \cdot 
        V(A\Chain(P_\pi)[j], \Chain(AP_\sigma)[n-j])  \ge  \frac{1}{j!(n-j)!}
        \, . \qedhere
    \]
\end{proof}

We can make more good use of the fact that $e_j(P,Q)$ can be interpreted as a mixed
volume.

\begin{thm}
    Let $(P,Q)$ be a double poset on $n$ elements. Then the sequence
    \[
        e_0(P,Q),e_1(P,Q),\dots,e_n(P,Q) 
    \] 
    is log-concave. 
\end{thm}
\begin{proof}
    The claim follows from the fact that $V(\Chain(P)[j],-\Chain(Q)[n-j])
    = \frac{e_j(P,Q)}{n!}$ and the Alexandrov--Fenchel inequality.  
\end{proof}

Note that if $Q = C_n$ is a chain, then $e_j(P,C_n) = \sum_J e(P|_J)$ where
the sum is over subsets of size $j$. If $Q = A_n$ is an anti-chain, then
$e_j(P,A_n) = (n-j)! e_j(P,C_n)$.  Other inequalities can be derived from, for
example, Minkowski's first inequality on mixed volumes to get
\emph{isoperimetric type} inequalities on $e_j(P,Q)$.

Notice that the sequence $e_j(P) := e_j(P,P)$ is log-concave and
\Def{palindromic}, that is, $e_i(P) = e_{n-i}(P)$ for all $0 \le i \le n$.
Let us briefly remark that the associated polynomial
\[
    E(P,t) \ := \ \sum_{j=0}^n e_j(P) t^{j} 
\]
is \emph{not} real-rooted, as is evident for an $n$-anti-chain $P = A_n$ with
$E(A_n,t) = n!(1 + t + \cdots + t^n)$. This example also shows $E(P,t)$ is
also not $\gamma$-positive; cf.~\cite{Christos}.

Finally, we can also derive lower and upper bounds from our geometric
considerations using the validity of Godberson's conjecture for anti-blocking
polytopes (Theorem~\ref{thm:god-for-closed-down}) and~\eqref{eqn:lower}.
However, these bounds can also be derived from simple combinatorial
considerations.

\begin{cor}
    Let $P$ be a poset on $n$ elements. Then 
    \[
        e(P) \ \le \ e_j(P) \ \le \ \binom{n}{j} e(P) \, .
    \]
    for all $0 \le j \le n$. The lower and upper bound is attained with
    equality precisely when $P$ is an anti-chain and chain, respectively. 
\end{cor}

Let us briefly remark on a different perspective on Sidorenko's inequality.
The \Def{weak (Bruhat) order} on $\SymGrp_n$ is given by 
\[
    \sigma \le^{wo} \pi \quad \Longleftrightarrow \quad 
    I(\sigma) \subseteq I(\pi) \, ,
\]
where $I(\pi) := \{ (\pi_i,\pi_j) : i < j, \pi_i > \pi_j\}$ is the
\Def{inversion set} of $\pi$. The minimum and maximum of $\SymGrp_n$ with
respect to the weak order are given by $\hat{0} = (1,2,\dots,n)$ and $\hat{1}
= (n,n-1,\dots,1)$. We write $[\sigma,\pi]_{wo}$ for the interval between
$\sigma$ and $\pi$. Bj\"{o}rner and Wachs~\cite{BW} proved the following
relation between the weak order and two-dimensional posets.

\begin{prop}
    Let $\pi,\sigma \in \SymGrp_n$. Then $\sigma$ is a linear extension of
    $P_\pi$ if and only if $\sigma \le^{wo} \pi$. In particular, $e(P_\pi) =
    \# [\hat{0},\pi]_{wo}$.
\end{prop}

Sidorenko's inequality (Theorem~\ref{thm:sidorenko}) is thus equivalent to

\begin{cor}
    For every $\pi \in \SymGrp_n$, we have 
    \[
        \# [\hat{0},\pi]_{wo} \cdot
        \# [\pi,\hat{1}]_{wo} \ \ge \ n! \, .
    \]
\end{cor}

There has been considerable interest in refined versions of this inequality;
see, for example, \cite{Wei,MPP} and, in particular,~\cite{GG}. It would be
very interesting to find an interpretation of Theorem~\ref{thm:gen_sidorenko}
in this setting.

\section{$C$-anti-blocking bodies and
decompositions}\label{sec:C-ab}

Anti-blocking bodies are defined relative to the cone $\RRn^n$.  In this
section, we develop anti-blocking bodies for other cones $C$ and we show that
analogous decompostion results hold. The section is self-contained and may be
read independently of the rest of the paper.

\subsection{$C$-anti-blocking bodies}
Let $C \subseteq \R^n$ be a full-dimensional and pointed polyhedral convex
cone. The \Def{polar cone} of $C$ is
\[
    C^\vee \ := \ \{ y \in \R^n : \iprod{y}{x} \ge 0 \text{ for all } x \in C
    \} \, , 
\]
which is again full-dimensional and pointed. We denote by $\preceq$ the
partial order on $\R^n$ induced by the \emph{dual} cone: for $x,y \in \R^n$ we
set $x \preceq y$ if $y - x \in C^\vee$. A convex body $K \subseteq C$ is
called \Def{$C$-anti-blocking} 
if for every $y \in K$ and $x \in C$ 
\[
    x \preceq y \quad \Longrightarrow \quad x \in K \, .
\]
If $C = \RRn^n$, then $C^\vee = C$ and $\RRn^n$-anti-blocking bodies are
exactly the usual (full-dimensional) anti-blocking bodies. In this subsection,
we develop the basics of $C$-anti-blocking bodies analogous to the treatment
in~\cite{Schrijver}. Although our results hold in greater generality, we make
the simplifying assumption that $C \subseteq C^\vee$ throughout.

\newcommand\Cdown[1]{\{#1\}^\downarrow_C}%
\newcommand\maxCvee{\mathrm{max}_{C^\vee}}%
For any compact set $U \subseteq C$
\[
    \Cdown{U} \ := \ C \cap (\conv(U)  - C^\vee ) 
\]
is a $C$-anti-blocking body. In particular, if $K$ is $C$-anti-blocking with
extreme points $V$, let $\maxCvee(V) \subseteq V$ be the points maximal with
respect to $\preceq$. Then $K  = \Cdown{\maxCvee(V)}$.

\begin{prop}
    The class of $C$-anti-blocking bodies is closed under intersection, convex
    hull, and Minkowski sums.
\end{prop}

We call $K$ \Def{proper} if for every $c \in C$ there is $\mu > 0$ such that
$\mu c \in K$. 
\begin{prop}\label{prop:Cab-rep}
    Let $K \subseteq C$ be a convex body. Then $K$ is a proper
    $C$-anti-blocking body if and only if there is a compact $W \subseteq C$
    not contained in a face of $C$ with
	\[
        K \ = \ \{ x \in C : \iprod{w}{x}  \le  1 \text{ for } w \in W \} \, .
	\]
\end{prop}
\begin{proof}
    To prove sufficency, first observe that $K$ is bounded. Indeed, if $W
    \subseteq C$ is not contained in a face, then there is $c \in \conv(W)$
    that is contained in the interior of $C$ Since $C \subseteq C^\vee$, this
    means that $\{ x \in C : \iprod{c}{x} \le 1\}$ is bounded and contains $K$.

    Now, if $y \in K$ and $x \preceq y$, there is $c
    \in C^\vee$ with $y = x + c$. Hence 
    \[
        \iprod{w}{x} \ = \ \iprod{w}{y} - \iprod{w}{c} \ \le \ 1,
    \]
    as $w \in C$ by assumption. For every $c \in C$, there is $\epsilon > 0$
    such that $\iprod{w}{\epsilon c} < 1$ for all $w \in W$. This shows that
    $K$ is a proper $C$-anti-blocking body.  Finally, for every $c \in C$,
    there is $\epsilon > 0$ such that $\iprod{w}{\epsilon c} < 1$ for all $w
    \in W$. This shows that $K$ is a proper $C$-anti-blocking body.

    To prove necessity, assume $K$ is properly $C$-anti-blocking. For every
    point $p \in \partial K$, there is a hyperplane $H = \{ x : \iprod{w}{x} =
    \delta \}$ supporting $K$ at $p$. Assume that $K \subseteq \{ x :
    \iprod{w}{x} \le \delta\}$. If $p \in \partial C$, then, since $K$ is
    proper, $H$ supports $C$ and we are done. If $p \in \partial K \setminus
    \partial C$, then we may assume $\delta = 1$. Now $p$ is contained in the
    interior of $C \subseteq C^\vee$. Thus, for every $c \in C^\vee$ there is
    $\epsilon > 0$ such that $p - \epsilon c \in K$.  Hence $1 \ge \iprod{w}{p
    - \epsilon c} = 1 - \epsilon \iprod{w}{c}$, which implies that
    $\iprod{w}{c} \ge 0$. Thus $w \in (C^\vee)^\vee = C$.
\end{proof}

Proposition~\ref{prop:Cab-rep} hints at a duality theory for $C$-anti-blocking
bodies akin to that of the usual anti-blocking bodies. Define
\[
    A_CK \ := \ \{ y \in C : \iprod{x}{y} \le 1 \text{ for all } x \in K \}
    \ = \ K\dual \cap C \, .
\]
\begin{lem}
    Let 
    \begin{align*}
        K \ &= \ \Cdown{U}  \ = \ \{ x \in C : \iprod{w}{x} \le 1
    \text{ for } w \in W\}\\
    \intertext{be a proper $C$-anti-blocking body with
    $U,W \subseteq C$. Then}
    A_CK \ &= \ \Cdown{W} \ = \  \{ x \in C : \iprod{u}{x}
    \le 1 \text{ for } u \in U\} 
    \end{align*}
    is also a $C$-anti-blocking body.  In particular $A_C A_C K = K$ for
    any $C$-anti-blocking body $K$.
\end{lem}
\begin{proof}
    If $K = \conv(V)$, then $A_CK = \{ y \in C : \iprod{y}{v} \le 1 \text{ for
    all } v \in V \}$. For $u,v \in C$ and $w \in C \subseteq C^\vee$, it
    holds that if $v \preceq u$, then $\iprod{w}{u} \le 1$ implies
    $\iprod{w}{v} \le 1$. Hence, we may replace $V$ with any superset of
    $\max_C(V)$ such as $U$. Proposition~\ref{prop:Cab-rep} now shows that
    $A_C K$ is $C$-anti-blocking.    

    Since $W \subseteq C$, it follows that $\Cdown{W}
    \subseteq A_C K$. For the reverse inclusion, we note that every $w \in
    A_C K$ is of the form $w = \sum_i \mu_i w_i - c$ with $c \in C^\vee$,
    $w_i \in W$, $\sum_i \mu_i \le 1$ and $\mu_i \ge 0$ for all $i$. That is,
    $w \in C \cap (\conv(W) - C^\vee) = \Cdown{W}$.
\end{proof}

It would be very interesting to prove an upper bound on the volume product
\[
    \vol K \cdot \vol A_CK \, .
\]
In order to adapt Meyer's proof of Saint-Raymond's inequality, we would need
that $K \cap F$ is $F$-anti-blocking for any inclusion-maximal face $F \subset
C$.  This, however, is the case if and only if the orthogonal projection of
$C$ onto the linear subspace spanned by $F$ is precisely $F$.  This rare
property only holds for certain linear images of $\RRn^n$.

\subsection{Decompositions}
The results of Lemma~\ref{lem:decomposition-sum} turn out to hold in the much
greater generality of $C$-anti-blocking bodies. For a cone $C$ and a face $F
\subseteq C$, let us denote by $F^\diamond = \{ c \in C^\vee : \iprod{c}{x} =
0 \text{ for all } x \in F \} \subseteq C^\vee$ the conjugate face. If $K$ is
$C$-anti-blocking, then define $K_F := K \cap F$. 

\begin{thm}\label{thm:Cab-dissect}
    Let $C$ be a full-dimensional and pointed polyhedral cone.  Let $K
    \subseteq \R^n$ be $C$-anti-blocking and $L \subseteq \R^n$ be
    $C^\vee$-anti-blocking.  Then
    \[
        K + (-L) \ = \ \bigcup_{F \subseteq C} K_F \times (-L_{F^\diamond})
        \quad \text{ and } \quad
        K \vee (-L) \ = \ \bigcup_{F \subseteq C} K_F \vee (-L_{F^\diamond})
        \, . 
	\]
    are dissections. In
    particular,
    \begin{align*}
	    V(K[j], -L[n-j]) \ &= \ {\binom{n}{j}}^{-1}\sum_{\dim F = j}
        \vol_j(K_F) \cdot \vol_{n-j} (L_{F^\diamond}) \\
	\vol(K \vee -L) \ &= \  \sum_{j=0}^n V(K[j],-L[n-j])
 \, .
	\end{align*}
\end{thm}

We start by recalling the following result. For a closed convex set $F
\subseteq \R^n$, let us write $\pi_F : \R^n \to F$ for the nearest point map.

\begin{lem}\label{lem:C-decomp}
	Let $C \subseteq \R^n$ be a polyhedral cone. Then
	\[
	    \R^n \ = \ \bigcup_F F + (- F^\diamond)
	\]
	where the union is over all faces $F \subseteq C$. Moreover, for distinct
	faces $F,G \subseteq C$, $(F -F^\diamond) \cap (G -G^\diamond)$ is a face
	of both.
\end{lem}
\begin{proof}
    For $p \in \R^n$ let $r = \pi_C(p)$ be the closest point in $C$. That is,
    $r$ is the unique minimizer of $x \mapsto \|p-x\|^2$ over $C$. It follows
    that this is the case if and only if $\iprod{r-p}{x} \ge 0$ for all $x \in
    C$. That is, $r-p \in C^\vee$. In particular, if $r$ is in the relative
    interior of the face $F$, then $r-p \in F^\diamond$. This yields the
    decomposition.
\end{proof}

\begin{cor} \label{cor:piC}
	Let $C \subseteq \R^n$ be a full-dimensional cone, $p \in \R^n$, and $r \in
	C$. Then $r = \pi_C(p)$ if and only if $r-p \in C^\vee$ with
	$\iprod{r-p}{r} = 0$.
\end{cor}

For a $C$-anti-blocking body $K$, let us write
\[
\widehat{K} \ = \ \{ x \in \R^n : \iprod{w}{x} \le 1 \text{ for }
w \in W \} \, ,
\]
where the $W \subseteq C$ is by Proposition~\ref{prop:Cab-rep}.

\begin{lem}\label{lem:Phat}
	Let $K \subseteq \R^n$ be a $C$-anti-blocking body and $p \in \widehat{K}$.
	Then $\pi_K(p) = \pi_C(p)$.
\end{lem}
\begin{proof}
	Let $r = \pi_K(p)$. Since the segment $[p,r]$ meets $K$ only in $r$, it
	follows that $r \in \partial C$. Thus the hyperplane $H = \{x :
	\iprod{r-p}{x} = \iprod{r-p}{r} \}$ supports $C$ and $r \in H \cap C$.
	This implies $r-p \in C^\vee$ and $\iprod{r-p}{r}
	= 0$ and hence $r = \pi_C(p)$ by Corollary~\ref{cor:piC}.
\end{proof}

\begin{proof}[Proof of Theorem~\ref{thm:Cab-dissect}]
    For $p \in K - L$, let $F \subseteq C$ be the face containing $r$ in the
    relative interior. Lemma~\ref{lem:C-decomp} states that $p \in F -
    F^\diamond$ and we claim that $p \in K_F - L_{F^\diamond}$. 
    
    It suffices to show $r \in K$ and $s := p-r \in -L$. This follows from
    Lemma~\ref{lem:Phat} once we verify $K - L \subseteq \widehat{K} \cap
    (-\widehat{L})$.  Since $L \subseteq C^\vee$ and $\widehat{K}$ is defined
    in terms of $w \in W \subseteq C$, we have 
    \[
        \iprod{w}{r - s} \ = \  \underbrace{\iprod{w}{r}}_{\le 1}  -
        \underbrace{\iprod{w}{s}}_{\ge 0}  \ \le \ 1 \, . 
    \]
    This shows $P - Q \subseteq \widehat{P}$ and the argument for
    $\widehat{Q}$ is analogous. 

    The proof of the decomposition for $K \vee (-L)$ is the same as in
    Lemma~\ref{lem:locally-ab-sum-decomp}: It follows from the fact that $K
    \vee (-L)$ is covered by the sets $\lambda K  + (1-\lambda)(-L)$ for
    $0 \le \lambda \le 1$.

    As for the (mixed) volume computations, observe that $\lambda K -  \mu L$
    for $\lambda, \mu > 0$ is decomposed by $\lambda K_F - \mu L_{F^\diamond}$
    for $F \subseteq C$. Since $F$ and $F^\diamond$ lie in orthogonal
    subspaces, it follows that $\vol(\lambda K_F - \mu L_{F^\diamond}) =
    \lambda^{\dim F} \mu^{n - \dim F} \vol(K_F) \vol(L_{F^\diamond})$.
    Collecting coefficients of $\vol(\lambda K - \mu L)$ yields the volume
    formulas.
\end{proof}

As a corollary to the proof, we note the following.

\begin{cor}
    Let $P$ be $C$-antiblocking and let $Q$ be $C^\vee$-antiblocking. Then
    \[
        P - Q \ = \ \widehat{P} \cap (-\widehat{Q}) \, .
    \]
\end{cor}
 \color{black}

\bibliographystyle{amsplain}
\addcontentsline{toc}{section}{\refname}\bibliography{AB-inequalities}

\end{document}